\documentclass[twoside]{article}
\usepackage{enumerate,mathrsfs,amsfonts,amsxtra,latexsym,amssymb,amsthm,verbatim,amsmath,graphicx,dsfont,extarrows}
\usepackage{mathpazo}
\usepackage[dvipdfm, colorlinks, linkcolor=green, anchorcolor=blue, citecolor=red]{hyperref}
\allowdisplaybreaks[4]
\catcode`@=12
\newtheorem{theorem}{Theorem}[section]

\newtheorem{lemma}[theorem]{Lemma}

\begin{document}
\abovedisplayskip=6pt plus 1pt minus 1pt \belowdisplayskip=6pt
plus 1pt minus 1pt
\thispagestyle{empty} \vspace*{-1.0truecm} \noindent
\vskip 10mm

\begin{center}{\large Integral positive (negative) quandle cocycle invariants are trivial for knots} \end{center}

\vskip 5mm
\begin{center}{Zhiyun Cheng \quad Hongzhu Gao\\
{\small School of Mathematical Sciences, Beijing Normal University
\\Laboratory of Mathematics and Complex Systems, Ministry of
Education, Beijing 100875, China
\\(email: czy@bnu.edu.cn \quad hzgao@bnu.edu.cn)}}\end{center}

\vskip 1mm

\noindent{\small {\small\bf Abstract} In this note we prove that for any finite quandle $X$ and any 2-cocycle $\phi\in Z^2_{Q\pm}(X; \mathds{Z})$, the cocycle invariant $\Phi^{\pm}_{\phi}(K)$ is trivial for all knots $K$.
\ \

\vspace{1mm}\baselineskip 12pt


\noindent{\small\bf MR(2010) Subject Classification} 57M27\ \ {\rm }}

\vskip 1mm

\vspace{1mm}\baselineskip 12pt

\let\thefootnote\relax\footnote{The authors are supported by NSFC 11301028 and the Fundamental Research Funds for Central Universities of China 105-105580GK.}

\section{Introduction}
A quandle \cite{Joy1982,Mat1984} is a set with a binary operation which satisfies some axioms motivated by variations on Reidemeister moves. Similar to the knot group, for every knot $K\subset S^3$ one can define the knot quandle $Q(K)$ \cite{Joy1982}. It is well known that $Q(K)$ is a powerful knot invariant. It characterizes the knot $K$ up to inversion. However in general the knot quandle is not easy to deal with, therefore it is more convenient to count the homomorphisms from $Q(K)$ to a fixed finite quandle, known as the quandle coloring invariant. More sophisticated invariants of this sort were introduced via quandle cohomology in \cite{Car2003}. More precisely, as a modification of the rack cohomology theory \cite{Fen1995,Fen1996}, J. S. Carter et al constructed the quandle cohomology theory, then they showed that with a given quandle 2-cocycle (3-cocycle) one can generalize the quandle coloring invariant to a state-sum invariant for knots (knotted surfaces). In particular when the 2-cocycle is a coboundary this state-sum invariant reduces to the quandle coloring invariant. In \cite{Che2014} we consider another boundary map of the chain complex and introduce the positive quandle cohomology theory. Similar to the quandle cohomology theory, with a given positive quandle 2-cocycle and 3-cocycle we can also define a state-sum invariant for knots and knotted surfaces respectively. Besides of the trivial quandle, it was proved that if we choose the dihedral quandle then the quandle cocycle invariants and positive quandle cocycle invariants are both trivial (i.e. reduce to the quandle coloring invariant) for knots \cite{Car2003,Che2014}. The main result of this note is an extension of this fact.
\begin{theorem}
Let $X$ be a finite quandle and $\phi\in Z^2_{Q-}(X; \mathds{Z})$, then for any knot $K$ the cocycle invariant $\Phi^{-}_{\phi}(K)$ is trivial.
\end{theorem}
\begin{theorem}
Let $X$ be a finite quandle and $\phi\in Z^2_{Q+}(X; \mathds{Z})$, then for any knot $K$ the cocycle invariant $\Phi^{+}_{\phi}(K)$ is trivial.
\end{theorem}

We remark that if we replace the coefficient group $\mathds{Z}$ with some other coefficient groups, for example $\mathds{Z}_2$, then the positive(negative) quandle cocycle invariants need not to be trivial. The readers are referred to \cite{Car2003} for some nontrivial cocycle invariants of trefoil and the figure eight knot. On the other hand, if we replace the knot $K$ with a link which has at least two components, then even using  the trivial quandle with two elements one can prove that the negative quandle cocycle invariants contain the information of linking numbers. For the positive quandle cocycle invariant, one can even use it to distinguish the Borromean ring from the 3-component trivial link. See \cite{Car2003} and \cite{Che2014} for more details.

The remainder of this note is arranged as follows. Section 2 contains the basic definitions of quandle, positive(negative) quandle (co)homology groups and the cocycle invariants. Section 3 and section 4 are devoted to the proof of Theorem 1.1 and Theorem 1.2 respectively.

\section{Quandle homology and cocycle invariants}
A \emph{quandle} $(X, \ast)$, is a set $X$ with a binary operation $(a, b)\rightarrow a\ast b$ satisfying the following axioms:
\begin{enumerate}
  \item For any $a\in X$, $a\ast a=a$.
  \item For any $b, c\in X$, there exists a unique $a\in X$ such that $a\ast b=c$.
  \item For any $a, b, c\in X$, $(a\ast b)\ast c=(a\ast c)\ast(b\ast c)$.
\end{enumerate}
Note that with a given binary operation $\ast$ we can define the dual operation $\ast^{-1}$ by
\begin{center}
$c\ast^{-1}b=a$ if $a\ast b=c$.
\end{center}
It is easy to observe that $(X, \ast^{-1})$ is also a quandle, called the \emph{dual quandle} of $(X, \ast)$. For convenience we simply denote $(X, \ast)$ by $X$. If $X$ satisfies the second and the third axioms, then we name it a \emph{rack}. Next we list some  examples of quandle:
\begin{itemize}
  \item Let $T_n=\{a_1, \cdots, a_n\}$ and $a_i\ast a_j=a_i$, we say $T_n$ is the trivial quandle with $n$ elements;
  \item Let $D_n=\{0, 1, \cdots, n-1\}$ and $i\ast j=2j-i$ (mod $n$), we say $D_n$ is the dihedral quandle of order $n$;
  \item Let $X$ be a conjugacy class of a group $G$ and $a\ast b=b^{-1}ab$, we say $X$ is a conjugation quandle.
\end{itemize}

Let $X$ be a quandle and $a, b$ two elements of $X$. We say $a$ and $b$ are of the same \emph{orbit} if $b$ can be obtained from $a$ by some right translations, i.e. there exist some elements $\{a_1, \cdots, a_n\}\subset X$ such that
\begin{center}
$b=(\cdots ((a\ast^{\varepsilon_1} a_1)\ast^{\varepsilon_2}a_2)\cdots )\ast^{\varepsilon_n} a_n$,
\end{center}
where $\varepsilon_i\in \{\pm1\}$. The orbit set of $X$ is denoted by Orb$(X)$, and the orbit that contains $a$ is denoted by Orb$(a)$. Obviously Orb$(a)$ is a subquandle of $X$. If $X$ has only one orbit then we say $X$ is \emph{connected}.

Note that the axioms for a quandle correspond to the three Reidemeister moves respectively. Hence by assigning each arc of a knot diagram with an element of $X$, such that the condition in Figure 1 is satisfied, then the number of colorings is a knot invariant, known as the quandle coloring invariant. Here by an arc we mean a part of the diagram from an undercrossing to the next undercrossing. In particular when $X=D_n$ this is exactly the well known Fox $n$-coloring invariant \cite{Fox1961}. For a fixed finite quandle $X$, we use Col$_{X}(K)$ to denote the quandle coloring invariant of $K$.
\begin{center}
\includegraphics{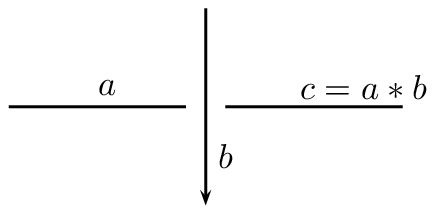}
\centerline{\small Figure 1: The coloring rule\quad}
\end{center}

In order to generalize the quandle coloring invariant, for a fixed coloring one can associate each crossing point with a (Boltzmann) weight such that the signed product of these weights are independent of the choice of the knot diagram. For this purpose we need to introduce the quandle (co)homology theory.

Let $X$ be a finite quandle and $C_n^R(X)$ the free abelian group generated by $n-$tuples $(a_1, \cdots, a_n)$, here each $a_i$ is an element of $X$. Let us consider the following two homomorphisms from $C_n^R(X)$ to $C_{n-1}^R(X)$:
\begin{flalign*}
&d_1(a_1, \cdots, a_n)=\sum\limits_{i=1}^n(-1)^i(a_1, \cdots, a_{i-1}, a_{i+1}, \cdots, a_n) \quad (n\geq2)\\
&d_2(a_1, \cdots, a_n)=\sum\limits_{i=1}^n(-1)^i(a_1\ast a_i, \cdots, a_{i-1}\ast a_i, a_{i+1}, \cdots, a_n) \quad (n\geq2)\\
&d_i(a_1, \cdots, a_n)=0 \quad (n\leq2, i=1, 2)
\end{flalign*}
It is routine to check that $d_1^2=d_2^2=d_1d_2+d_2d_1=0$, it follows that we have two chain complexes $\{C_n^R(X), \partial^+=d_1+d_2\}$ and $\{C_n^R(X), \partial^-=d_1-d_2\}$. Define $C_n^D(X)$ $(n\geq2)$ to be the free abelian group generated by $n$-tuples $(a_1, \cdots, a_n)$ with $a_i=a_{i+1}$ for some $1\leq i\leq n-1$. When $n\leq1$ we define $C_n^D(X)=0$. It is not difficult to check that $\{C_n^D(X), d_i\}$ is a sub-complex of $\{C_n^R(X), d_i\}$ $(i=1, 2)$, from this we conclude that  $\{C_n^D(X), \partial^{\pm}\}$ is a sub-complex of $\{C_n^R(X), \partial^{\pm}\}$. Let $C_n^Q(X)=C_n^R(X)/C_n^D(X)$ and $G$ an abelian group, then we can define the chain complexes and cochain complexes $(W\in\{R, D, Q\})$
\begin{itemize}
  \item $C_*^{W\pm}(X; G)=\{C_*^W(X), \partial^{\pm}\}\bigotimes G$, \quad $\partial^\pm=\partial^\pm\bigotimes$ id;
  \item $C_{W\pm}^{\ast}(X; G)=$Hom$(\{C_*^W(X), \partial^{\pm}\}, G)$, \quad $\delta_\pm=$Hom$(\partial^\pm$, id).
\end{itemize}
Now we define the \emph{positive quandle (co)homology groups} and \emph{negative quandle (co)homology groups} of $X$ with coefficient $G$ to be the (co)homology groups of $C_*^{Q\pm}(X; G)$ $(C_{Q\pm}^*(X; G))$. The positive(negative) rack (co)homology groups and positive(negative) degeneration (co)homology groups can be defined in an analogous way. We remark that the negative quandle (co)homology groups are exactly the well studied quandle (co)homology groups defined by J. S. Carter et al in \cite{Car2003}. In order to distinguish it from the positive quandle (co)homology groups we will call it negative quandle (co)homology groups throughout this note.

Now we turn to the definition of cocycle invariants. Let $K$ be a knot diagram and $X$ a fixed finite quandle, we use $\rho$ to denote a coloring of $K$ by $X$. It is well known that the regions of $R^2-K$ have a checkerboard fashion coloring. Without loss of generality we assume the unbounded region has the white color. Then we can assign each crossing $\tau$ of $K$ with a sign $\epsilon(\tau)$ according to Figure 2.
\begin{center}
\includegraphics{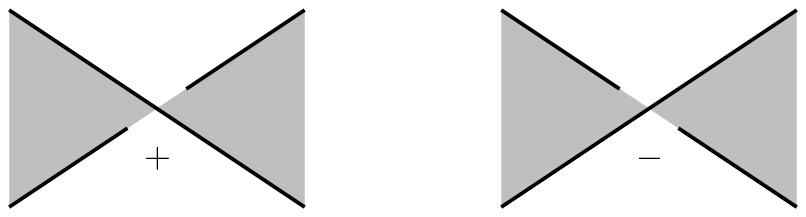}
\centerline{\small Figure 2: The signs of crossings\quad}
\end{center}
Assume the over-arc and under-arcs of $\tau$ are colored by $b$ and $a, a\ast b$ respectively, see Figure 1. Then with a given negative quandle 2-cocycle $\phi$ we can define a weight $W_{\phi}(\tau, \rho)$ associate to $\tau$ as
\begin{center}
$W_{\phi}^-(\tau, \rho)=\phi(a, b)^{w(\tau)}$,
\end{center}
here $w(\tau)$ denotes the writhe of the crossing. The \emph{negative quandle 2-cocycle invariant} of $K$ is defined to be
\begin{center}
$\Phi_{\phi}^-(K)=\sum\limits_{\rho}\prod\limits_{\tau}W_{\phi}^-(\tau, \rho)\in \mathds{Z}G$,
\end{center}
where $\rho$ runs all colorings of $K$ by $X$ and $\tau$ runs all crossing points of the diagram. If $\phi$ is a positive quandle 2-cocycle, we need to make a little modification on the definition of $W_{\phi}^-(\tau, \rho)$. By replacing $w(\tau)$ with $\epsilon(\tau)$, we define
\begin{center}
$W_{\phi}^+(\tau, \rho)=\phi(a, b)^{\epsilon(\tau)}$,
\end{center}
and the \emph{positive quandle 2-cocycle invariant} of $K$ can be defined by
\begin{center}
$\Phi_{\phi}^+(K)=\sum\limits_{\rho}\prod\limits_{\tau}W_{\phi}^+(\tau, \rho)\in \mathds{Z}G$,
\end{center}
similarly here $\rho$ runs all colorings of $K$ by $X$ and $\tau$ runs all crossing points of $K$.
\begin{theorem}{\emph{\cite{Car2003}}}
The negative quandle 2-cocycle invariant $\Phi_{\phi}^-(K)$ is invariant under Reidemeister moves. Moreover if a pair of negative quandle 2-cocycles $\phi_1$ and $\phi_2$ are cohomologous, then $\Phi_{\phi_1}^-(K)=\Phi_{\phi_2}^-(K)$.
\end{theorem}
\begin{theorem}{\emph{\cite{Che2014}}}
The positive quandle 2-cocycle invariant $\Phi_{\phi}^+(K)$ is invariant under Reidemeister moves. Moreover if a pair of positive quandle 2-cocycles $\phi_1$ and $\phi_2$ are cohomologous, then $\Phi_{\phi_1}^+(K)=\Phi_{\phi_2}^+(K)$.
\end{theorem}

We remark that with a positive(negative) quandle 3-cocycle one can also define a state-sum invariant for knotted surfaces \cite{Car2003,Che2014}, but we will not discuss it in this note. On the other hand, in a recent paper \cite{Kam2014} of S. Kamada, V. Lebed and K. Tanaka, by investigating the relation between the Alexander numbering and $\epsilon(\tau)$ they studied the relation between positive quandle 2-cocycle invariants and the twisted quandle cocycle invariants, which was proposed by J. S. Carter, M. Elhamdadi and M. Saito in \cite{Car2002}. See \cite{Kam2014} for more details.

\section{The proof of Theorem 1.1}
From now on the coefficient group $G$ is supposed to be $\mathds{Z}$, and as usual we will write $H_{Q\pm}^*(X, \mathds{Z})$ simply by $H_{Q\pm}^*(X)$. We begin the proof with a simple lemma.
\begin{lemma}
If $\phi$ has finite order in $H_{Q-}^2(X)$, then $\Phi_{\phi}^-(K)$ is trivial for any knot $K$.
\end{lemma}
\begin{proof}
Assume $k\phi=0\in H_{Q-}^2(X)$, then by Theorem 2.1 we have $\prod\limits_\tau W_{k\phi}^-(\tau, \rho)=0$ for each coloring $\rho$. In other words, $\prod\limits_\tau k\phi(a, b)^{w(\tau)}=k(\prod\limits_\tau\phi(a, b)^{w(\tau)})=0$. Since we are working with coefficient $\mathds{Z}$, it implies $\prod\limits_\tau\phi(a, b)^{w(\tau)}=0$.
\end{proof}

Lemma 3.1 tells us that it suffices to consider the free part of $H_{Q-}^2(X)$. The following key lemma mainly follows from P. Etingof and M. Gra\~{n}a's result \cite{Eti2003}.
\begin{lemma}
If $X$ is a connected quandle, then $\Phi_{\phi}^-(K)$ is trivial for all knots $K$.
\end{lemma}
\begin{proof}
According to the definition of negative quandle cohomology group, we have a long exact sequence
\begin{center}
$\cdots\rightarrow H_n^{D-}(X)\rightarrow H_n^{R-}(X)\rightarrow H_n^{Q-}(X)\rightarrow H_{n-1}^{D-}(X)\rightarrow\cdots$.
\end{center}
It was proved by Litherland and Nelson \cite{Lit2003} that this sequence is split into short exact sequences, i.e. there exists a short exact sequence
\begin{center}
$0\rightarrow H_n^{D-}(X)\rightarrow H_n^{R-}(X)\rightarrow H_n^{Q-}(X)\rightarrow 0$.
\end{center}
On the other hand, in \cite{Eti2003} P. Etingof and M. Gra\~{n}a calculated the Betti numbers of negative rack cohomology groups of $X$ are dim$H_{R-}^n(X; \mathds{Q})=|\text{Orb}(X)|^n$. In particular, when $X$ is connected we have dim$H_{R-}^2(X; \mathds{Q})=1$.

Now we turn to $H_2^{D-}(X)$. Since
\begin{flalign*}
&\partial^-(a, a)=-((a)-(a))+((a)-(a\ast a))=0\\
&\partial^-(a, a, b)=-(a, b)+(a, b)+(a, b)-(a\ast a, b)-(a, a)+(a\ast b, a\ast b)=-(a, a)+(a\ast b, a\ast b)\\
&\partial^-(a, b, b)=-(b, b)+(b, b)+(a, b)-(a\ast b, b)-(a, b)+(a\ast b, b\ast b)=0,
\end{flalign*}
it follows that $H_2^{D-}(X)=\mathds{Z}[\text{Orb}(X)]$, the free abelian group generated by the Orb$(X)$. Recall that $X$ is connected, hence $H_2^{D-}(X)=\mathds{Z}$.

According to the discussion above, together with the universal coefficient theorem \cite{Car2001}, we conclude that $H_{Q-}^2(X; \mathds{Q})=0$. The result follows.
\end{proof}

Now we give the proof of Theorem 1.1.
\begin{proof}
Let $\rho:Q(K)\rightarrow X$ denote a coloring of $K$ by $X$, it is sufficient to show that $\prod\limits_{\tau}W_{\phi}^-(\tau, \rho)=0$. Since $K$ has only one component, it is obvious that $\rho(Q(K))$ belong to the same orbit of $X$. Choose an element $a\in \rho(Q(K))$, since $\rho(Q(K))\subset \text{Orb}(a)$, therefore $\rho$ represents a coloring of $K$ by Orb$(a)$. Let $\phi$ be a negative quandle 2-cocycle of $X$, then it satisfies the negative quandle 2-cocycle condition
\begin{center}
$\phi(x, z)-\phi(x\ast y, z)-\phi(x, y)+\phi(x\ast z, y\ast z)=0$ for $\forall x, y, z\in X$.
\end{center}
Now we define a map $\phi':\text{Orb}(a)\times \text{Orb}(a)\rightarrow \mathds{Z}$ by
\begin{center}
$\phi'(x, y)=\phi(x, y)$ for $\forall x, y\in\text{Orb}(a)$.
\end{center}
Obviously $\phi'\in Z_{Q-}^2(\text{Orb}(a))$. Then we have
\begin{center}
$\prod\limits_{\tau}W_{\phi}^-(\tau, \rho)=\prod\limits_{\tau}W_{\phi'}^-(\tau, \rho)=0$,
\end{center}
the last equality comes from the fact that Orb$(a)$ is connected and Lemma 3.2. The proof is finished.
\end{proof}

\section{The proof of Theorem 1.2}
In this section we give the proof of Theorem 1.2, which is quite different from that of Theorem 1.1.
\begin{lemma}
Let $X$ be a finite quandle, $a$ an element of $X$ and $\phi\in Z_{Q+}^2(X)$, suppose $\rho$ is a coloring of $K$ by $X$, then $\rho\ast a$ is also a coloring of $K$ by $X$. Here $\rho\ast a$ denotes the coloring which assigns $x\ast a$ to an arc if $\rho$ assigns $x$ to it. Moreover we have $\prod\limits_{\tau}W_{\phi}^+(\tau, \rho)+\prod\limits_{\tau}W_{\phi}^+(\tau, \rho\ast a)=0$.
\end{lemma}
\begin{proof}
$\rho\ast a$ is a coloring follows directly from the third axiom of quandle.

Now we show that $\prod\limits_{\tau}W_{\phi}^+(\tau, \rho)+\prod\limits_{\tau}W_{\phi}^+(\tau, \rho\ast a)=0$. Because $\phi\in Z_{Q+}^2(X)$, thus $\phi$ satisfies the positive quandle 2-cocycle condition
\begin{center}
$-\phi(y, z)-\phi(y, z)+\phi(x, z)+\phi(x\ast y, z)-\phi(x, y)-\phi(x\ast z, y\ast z)=0$ for $\forall x, y, z\in X$.
\end{center}
It follows that
\begin{center}
$\phi(x, y)+\phi(x\ast z, y\ast z)=-\phi(y, z)-\phi(y, z)+\phi(x, z)+\phi(x\ast y, z)$.
\end{center}
Replacing $z$ with $a$ we have
\begin{center}
$\phi(x, y)+\phi(x\ast a, y\ast a)=-\phi(y, a)-\phi(y, a)+\phi(x, a)+\phi(x\ast y, a)$.
\end{center}
One computes
\begin{flalign*}
&\prod\limits_{\tau}W_{\phi}^+(\tau, \rho)+\prod\limits_{\tau}W_{\phi}^+(\tau, \rho\ast a)\\
=&\prod\limits_{\tau}\phi(x, y)^{\epsilon(\tau)}+\prod\limits_{\tau}\phi(x\ast a, y\ast a)^{\epsilon(\tau)}\\
=&\prod\limits_{\tau}(\phi(x, y)+\phi(x\ast a, y\ast a))^{\epsilon(\tau)}\\
=&\prod\limits_{\tau}(-\phi(y, a)-\phi(y, a)+\phi(x, a)+\phi(x\ast y, a))^{\epsilon(\tau)}
\end{flalign*}
The second equality holds as we are working with coefficient $\mathds{Z}$, for the same reason $\prod\limits_{\tau}(-\phi(y, a)-\phi(y, a)+\phi(x, a)+\phi(x\ast y, a))^{\epsilon(\tau)}$ can be rewritten as $\sum\limits_{\tau}\epsilon(\tau)(-\phi(y, a)-\phi(y, a)+\phi(x, a)+\phi(x\ast y, a))$.

In order to finish the proof we need to show that $\sum\limits_{\tau}\epsilon(\tau)(-\phi(y, a)-\phi(y, a)+\phi(x, a)+\phi(x\ast y, a))=0$. Notice that if the over-arc and under-arcs of $\tau$ are colored by $y$ and $x, x\ast y$ respectively, then the contribution that comes from $\tau$ equals $\epsilon(\tau)(-\phi(y, a)-\phi(y, a)+\phi(x, a)+\phi(x\ast y, a))$. We can divide this contribution into three parts: $-2\epsilon(\tau)\phi(y, a), \epsilon(\tau)\phi(x, a), \epsilon(\tau)\phi(x\ast y, a)$, and assign them to the over-arc and two under-arcs respectively. If the diagram is alternating, then all $\epsilon(\tau)$ are the same. In this case it is not difficult to observe that the contributions from all crossings will cancel out. The proof of the non-alternating case is analogous to the alternating case. In fact it suffices to notice that if an arc crosses several crossings as the over-arc, then the sign $\epsilon$ of these crossings are alternating. It follows that in this general case we still have $\sum\limits_{\tau}\epsilon(\tau)(-\phi(y, a)-\phi(y, a)+\phi(x, a)+\phi(x\ast y, a))=0$.
\end{proof}

\begin{lemma}
$\prod\limits_{\tau}W_{\phi}^+(\tau, \rho)=\prod\limits_{\tau}W_{\phi}^+(\tau, \rho\ast a)$.
\end{lemma}
\begin{proof}
Let $K$ be a knot diagram. If we place a unknot diagram $U$ near $K$ which has no intersection with $K$ then we obtain a two-component link diagram $L=K\cup U$. Obviously $K$ and $L$ have the same positive quandle 2-cocycle invariants if $U$ is always colored by $a$. Let us move $K$ into the circle $U$ via some second and third Reidemeister moves, when $K$ meets $U$ during the move we always put $U$ over $K$. Now we obtain a new link diagram $L'$. By checking the variations of coloring under Reidemeister moves one will find that if an arc of $L$ is colored by $x$ then it will be colored by $x\ast a$ in $L'$. See Figure 3.
\begin{center}
\includegraphics{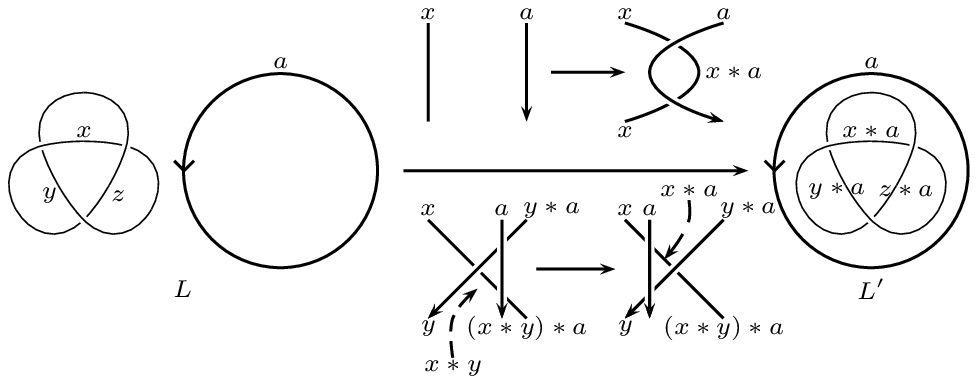}
\centerline{\small Figure 3: Transformation between $L$ and $L'$\quad}
\end{center}
Let $K$ be a knot diagram and $\rho$ a coloring of $K$ by $X$. If $K'$ can be obtained from $K$ by a sequence of Reidemeister moves then these Reidemeister moves will induce a coloring $\rho'$ on $K'$. It was proved in \cite{Che2014} that $\rho$ and $\rho'$ offer the same contribution to the positive quandle 2-cocycle invariants. Together with the discussion above we conclude that $\prod\limits_{\tau}W_{\phi}^+(\tau, \rho)=\prod\limits_{\tau}W_{\phi}^+(\tau, \rho\ast a)$.
\end{proof}

Theorem 1.2 follows directly from Lemma 4.1 and Lemma 4.2.

\textbf{Remark} There is another proof of Lemma 4.2 without needing to introduce the unknot diagram $U$ \cite{Mos2006}. Choose an arc $\lambda$ of $K$, without loss of generality we assume that $\lambda$ is colored by $a$. Let us take the first Reidemeister move $\Omega_1$ on $\lambda$ and slide the remainder of $K$ through the small loop created by the Reidemeister move. After taking $\Omega_1^{-1}$ on $\lambda$ we obtain the coloring $\rho\ast a$ on $K$. On the other hand we want to remark that if we replace the knot $K$ by a link $L$ and replace the coefficient group $\mathds{Z}$ by an abelian group without 2-torsion, Theorem 1.2 still holds.

\section*{Acknowledgement}
The authors would like to thank Takefumi Nosaka and Daniel Moskovich for pointing \cite{Eti2003} and \cite{Mos2006} to us respectively.

\end{document}